\documentclass[11pt]{amsart}
\usepackage{amssymb,amsmath,amsthm,amscd}
\usepackage{graphicx}
\usepackage{hyperref}
\usepackage{color}
\usepackage[usenames,dvipsnames,svgnames,table]{xcolor}
\usepackage{tikz}
\usepackage{comment}
\usepackage{epstopdf}
\usepackage{float}
\usepackage{tikz-cd}
\usepackage{psfrag}
\newtheorem{lemma}{Lemma}[section]
\newtheorem{theorem}[lemma]{Theorem}

\theoremstyle{question}

\theoremstyle{definition}
\newtheorem*{definition}{Definition}
\theoremstyle{remark}

\newcommand{\C}{\mathbb{C}}

\newcommand{\Z}{\mathbb{Z}}

\renewcommand{\epsilon}{\varepsilon}
\renewcommand{\phi}{\varphi}
\renewcommand{\theta}{\vartheta}
\def\C{\mathbb{C}}
\def\H{\mathbb{H}}

\title[Rigidity of parabolic germs]{A rigidity result for some parabolic germs}
\author[L.~Lomonaco]{Luna Lomonaco}
\address{Instituto de Matem\'atica e Estat\'istica da Universidade de S\~ao Paulo,
Rua do Mat\~ao, 1010 - CEP 05508-090 - S\~ao Paulo - SP}
\email{lluna@ime.usp.br }

\author[S.~Mukherjee]{Sabyasachi Mukherjee}
\address{Institute for Mathematical Sciences, Stony Brook University, NY, 11794, USA}
\email{sabya@math.stonybrook.edu}

\date{\today}

\begin{document}
\begin{abstract}
The goal of this article is to prove a rigidity result for unicritical polynomials with parabolic cycles. More precisely, we show that if two unicritical polynomials have conformally conjugate parabolic germs, then the polynomials are affinely conjugate.
\end{abstract} 

\maketitle

\tableofcontents

\section{Introduction}

The study of conformal conjugacy classes of holomorphic germs fixing the origin is a classical area of research in complex analysis. Of particular interest (and rich structure) is the set of tangent-to-identity germs
\begin{equation*}
\mathrm{Diff}^{+1}(\C,0):=\{f(z)=z+az^{n+1}+\cdots; n\geq 1, a\neq 0, f\in\C\{z\}\}.
\end{equation*}
Such a germ is also called a parabolic germ of multiplier $1$ (this is the terminology we will use in this paper). For a parabolic germ of multiplier $1$ as above, the integer $n+1$ is called the multiplicity of the fixed point $0$. By the work of Ecalle and Voronin \cite{Ec1,Vor}, there exists an infinite-dimensional family of conformally different parabolic germs of multiplier $1$ with $0$ being a fixed point of multiplicity $n+1$.

In holomorphic dynamics (in particular, in the iteration theory of polynomials), parabolic fixed points play a crucial role. A fixed point $\hat{z}$ of a polynomial $P(z)$ (in one complex variable) is called parabolic if $P'(\hat{z})$ is a $q$-th root of unity. Evidently, the polynomial $P^{\circ q}$ (where $P^{\circ q}$ stands for the $q$-th iterate of $P$) satisfies $(P^{\circ q})'(\hat{z})=1$. Conjugating $P^{\circ q}$ by a translation that sends $\hat{z}$ to $0$, we obtain a polynomial parabolic germ of multiplier $1$ fixing the origin. Thus, a parabolic fixed point of a polynomial determines a polynomial element of $\mathrm{Diff}^{+1}(\C,0)$. Since polynomials are global objects, it is quite reasonable to expect that if the parabolic fixed points of two polynomials determine conformally conjugate elements of $\mathrm{Diff}^{+1}(\C,0)$, then the global dynamics of the two polynomials are intimately related (compare \cite[\S 3]{CEP}). The principal goal of this paper is to formalize and prove this heuristic idea for unicritical polynomials. 

Any unicritical polynomial of degree $d\geq 2$ can be affinely conjugated to a map of the form $f_c(z)=z^d+c$ (an affine conjugacy respects the conformal dynamics of a polynomial). The \emph{filled Julia set} $K(f_c)$ is defined as the set of all points which remain bounded under all iterations of $f_c$. The boundary of the filled Julia set is defined to be the \emph{Julia set} $J(f_c)$. The degree $d$ multibrot set $\mathcal{M}_d$ is the connectedness locus of degree $d$ unicritical polynomials; i.e. $$\mathcal{M}_d:=\{c\in\C: K(f_c)\ \mathrm{is\ connected}\}.$$ The multibrot set of degree $2$ is called the Mandelbrot set.

A parameter $c\in\mathcal{M}_d$ is called \emph{hyperbolic} if the forward orbit of the unique critical point $0$ of $f_c$ converges to an attracting cycle. A connected component $H$ of the set of all hyperbolic parameters is called a \emph{hyperbolic component} of $\mathcal{M}_d$. Every parameter on the closure of a hyperbolic component $H$ has a unique non-repelling cycle (uniqueness is a consequence of unicriticality). The multiplier of this unique non-repelling cycle defines a $(d-1)$-fold map (which is holomorphic in the interior of $H$ and continuous up to $\partial H$) from the closure of $H$ onto the closure of the unit disk in the complex plane. For the Mandelbrot set, i.e. for $d=2$, this map yields an actual homeomorphism. This map is called the \emph{multiplier map}, and is denoted by $\lambda_{\overline{H}}$. 

A parameter $c$ of $\mathcal{M}_d$ is called a parabolic parameter if $f_c$ has a periodic cycle with multiplier a root of unity. Note that every parabolic cycle of a polynomial attracts the forward orbit of at least one critical point \cite[Theorem 10.15]{M1new}. Since the polynomials $f_c$ have a unique critical point, it follows that $f_c$ can have at most one parabolic cycle. Every parabolic parameter of $\mathcal{M}_d$ lies on the boundary of some hyperbolic component $H$. A parabolic parameter $c$ is called the \emph{root} of a hyperbolic component $H$ of period greater than $1$ if $\lambda_{\overline{H}}(c)=1$ and the parabolic cycle of $f_c$ disconnects $J(f_c)$ (the period $1$ hyperbolic component
is exceptional in the sense that, in the dynamical plane of its root, the parabolic cycle does not disconnect the Julia set). On the other hand, a parabolic parameter $c$ is called a \emph{co-root} of a hyperbolic component $H$ if $\lambda_{\overline{H}}(c)=1$ and the parabolic cycle of $f_c$ does not disconnect $J(f_c)$. There are exactly one root and $(d-2)$ co-roots on the boundary of every hyperbolic component $H$ (of period greater than $1$) of $\mathcal{M}_d$ (compare \cite[Theorem 1]{EMS}). In particular, there is no co-root in the Mandelbrot set. 

For $i=1, 2$, let $c_i$ be a root or co-root point of a hyperbolic component $H_i$ of period $n_i$ of $\mathcal{M}_d$, and $z_i$ be the characteristic parabolic point of $f_{c_i}$, i.e. the parabolic periodic point on the boundary of the critical value Fatou component (where the critical value Fatou component is the Fatou component containing the critical value). It is worthwhile to note that under the above assumptions, $(f_{c_i}^{\circ n_i})'(z_i)=1$; i.e. the restriction of $f_{c_i}^{\circ n_i}$ to a neighborhood of $z_i$ determines an element of $\mathrm{Diff}^{+1}(\C,0)$. The following theorem is the main result of this paper.

\begin{theorem}[Parabolic Germs Determine Parabolic Parameters]\label{Rigidity_Satellite}
Suppose that there exist small neighborhoods $N_1$ and $N_2$ of $z_1$ and $z_2$ (in the dynamical planes of $c_1$ and $c_2$ respectively) such that $f_{c_1}^{\circ n_1}\vert_{N_1}$ and $f_{c_2}^{\circ n_2}\vert_{N_2}$ are conformally conjugate (i.e. $f_{c_1}^{\circ n_1}\vert_{N_1}$ and $f_{c_2}^{\circ n_2}\vert_{N_2}$ determine conformally conjugate elements of $\mathrm{Diff}^{+1}(\C,0)$). Then $f_{c_1}$ and $f_{c_2}$ are affinely conjugate. 
\end{theorem}

Note that $f_{c_1}$ and $f_{c_2}$ are affinely conjugate if and only if $c_2/c_1$ is a $(d-1)$-st root of unity. Therefore, the above theorem states that the conformal conjugacy class of the parabolic germ restriction of a suitable iterate of $f_{c}$ determines $c$ uniquely up to the action of a finite rotation group.

The paper is organized as follows. In Section \ref{preliminars}, we recall some general facts about parabolic dynamics and parabolic parameters in the multibrot set $\mathcal{M}_d$. We review the theory
of parabolic-like maps (a parabolic analogue of polynomial-like maps introduced by the first author in \cite{Lu}) in Section \ref{parlike}. Section \ref{proof} is devoted to the proof of Theorem \ref{Rigidity_Satellite}. There are two essential steps in the proof of the theorem. In Lemma  \ref{rigidity_PL_maps}, we prove a rigidity theorem for parabolic-like maps. Subsequently in Lemma \ref{equivalence}, we prove a local-to-global principle to the effect that a local conformal conjugacy between certain parabolic germs can be promoted to a conformal conjugacy between suitable parabolic-like maps. The main theorem now follows by combining these two lemmas.

\section{Preliminaries}\label{preliminars}
Let $f(z)=z+az^{n+1}+\cdots$ be a parabolic germ of multiplier $1$. The local dynamics of such a germ is well-understood. Let us briefly review the situation for completeness. There are $2n$ open sectors $V_1^+, V_1^-, \cdots, V_n^+, V_n^-$ based at $0$ such that they cover a deleted neighborhood of the origin. For each $i=1, 2, \cdots, n$, we have $f(V_i^+)\subset V_i^+$ and the forward orbit (under $f$) of each point in $V_i^+$ converges to the parabolic fixed point $0$ (compare Figure \ref{petals}). These sectors are called \emph{attracting petals}. Moreover, there exists a conformal embedding $\psi^{\mathrm{att}, i}$ of $V_i^+$ into $\C$ such that the image contains a right-half plane and $\psi^{\mathrm{att}, i}$ conjugates $f$ to translation by $+1$. The maps $\psi^{\mathrm{att}, i}$ are called attracting Fatou coordinates. On the other hand, for each $i=1, 2, \cdots, n$, we have $f^{-1}(V_i^-)\subset V_i^-$ and the backward orbit (under $f$) of each point in $V_i^-$ converges to $0$. These sectors are called \emph{repelling petals}. There exists a a conformal embedding $\psi^{\mathrm{rep}, i}$ of $V_i^-$ into $\C$ such that the image contains a right-half plane and $\psi^{\mathrm{rep}, i}$ conjugates $f^{-1}$ to translation by $+1$. The maps $\psi^{\mathrm{rep}, i}$ are called repelling Fatou coordinates. These coordinates are unique up to addition of a complex constant. For a more rigorous description of the dynamics of parabolic germs and construction of Fatou coordinates, see \cite[\S 10]{M1new}\cite[\S 2.3]{FL}.

\begin{figure}[ht!]
\includegraphics[scale=0.12]{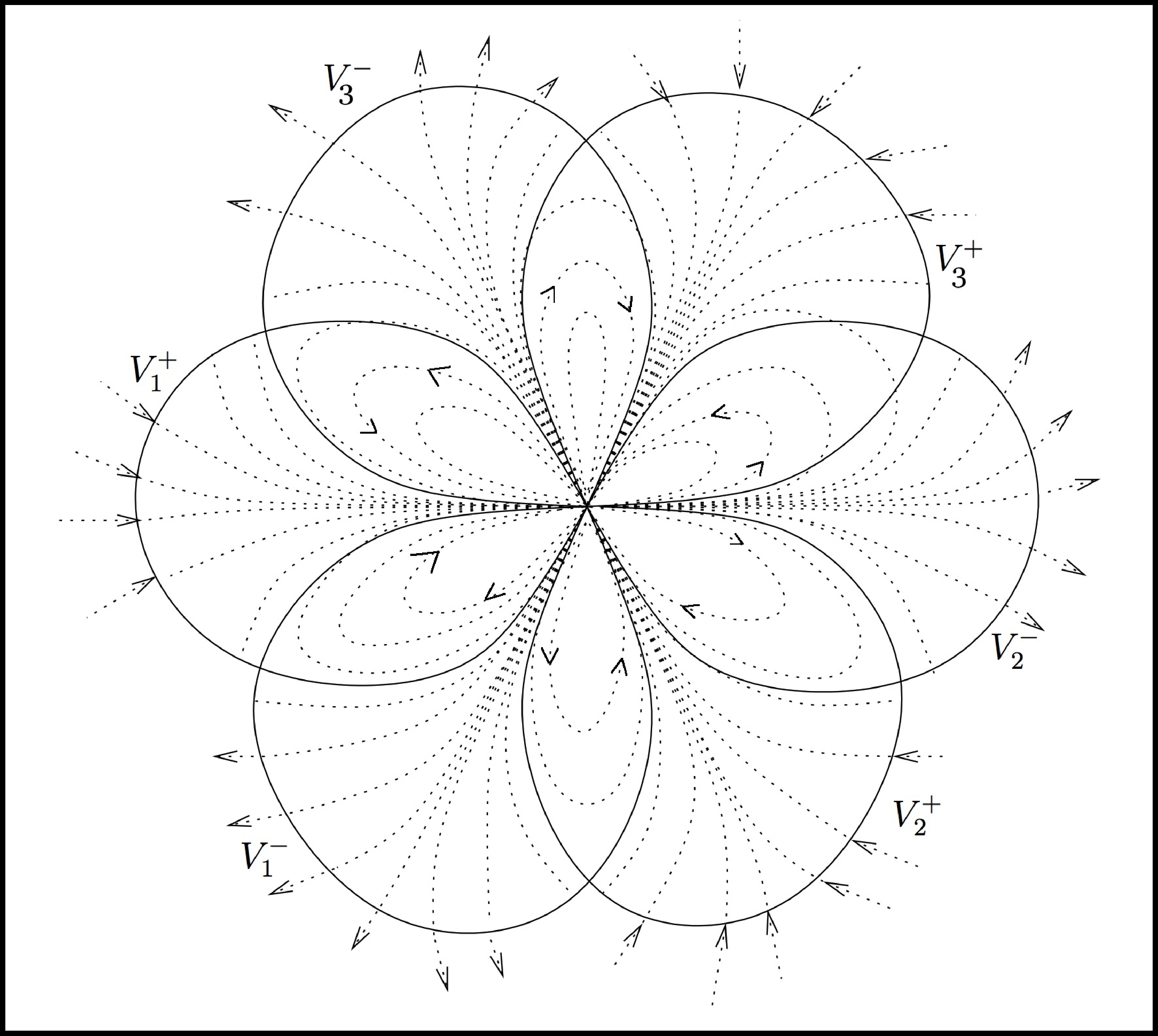}
\caption{Dynamics of a parabolic germ with three attracting and three repelling petals.}
\label{petals}
\end{figure}

For a general parabolic germ, these various Fatou coordinates do not agree on their common domains of definition; namely, on the intersection of two adjacent attracting and repelling petals. This leads to the collection of horn maps $$\lbrace \psi^{\mathrm{att},i}\circ\left(\psi^{\mathrm{rep},i}\right)^{-1}, \psi^{\mathrm{att},i+1}\circ\left(\psi^{\mathrm{rep},i}\right)^{-1}\rbrace_{i\in\Z/n\Z},$$ that record the difference between two adjacent attracting and repelling Fatou coordinates. The importance of horn maps stems from the fact that they are complete conformal conjugacy invariants of parabolic germs \cite{Ec1,Vor}. We refer the readers to \cite[\S 2]{BE} for a precise definition of horn maps. 

If a parabolic germ is obtained as a restriction of a globally defined polynomial, then the attracting and (the inverse of the) repelling Fatou coordinates have natural maximal domains of definition. The extensions are obtained by iterating the dynamics. Hence the associated horn maps extend as ramified coverings to certain natural maximal domains of definition. Definitions of these extended horn maps and their mapping properties can be found in \cite[\S 2.5]{BE}. The main property of extended horn maps that we will use in this paper is that they have finitely many critical values and these critical values are completely determined by the conformal positions of the critical points of the polynomial \cite[Proposition 4]{BE}\cite{Eps}. In particular, if the polynomial is unicritical, then each of these extended horn maps has a unique critical value.

Let us now recall some basic facts about the parabolic parameters of the multibrot set $\mathcal{M}_d$. Every parabolic parameter $c$ (i.e. a parameter $c$ such that $f_c$ has a periodic orbit of multiplier a root of unity) of $\mathcal{M}_d$ is either  the root or a co-root of a unique hyperbolic component $H$. We will denote the characteristic parabolic point of $f_c$ by $z_c$; i.e. $z_c$ is the unique parabolic periodic point of $f_c$ that lies on the boundary of the critical value Fatou component. A parabolic parameter $c$ is a root (respectively a co-root) of a hyperbolic component of period
greater than one if $z_c$ is a cut point of $K(f_c)$ (respectively, $z_c$ is not a cut point of $K(f_c)$), where $z_c$ is called a cut point of $K(f_c)$ if $K(f_c)\setminus\{z_c\}$ is disconnected. In the parameter plane, a corresponding dichotomy holds: a root parameter of a hyperbolic component of period $n>1$
is a cut point of $\mathcal{M}_d$, whereas a co-root parameter is not.

If $c$ is a co-root of a hyperbolic component $H$ of period $n$, then $f_c$ has a parabolic cycle of period $n$ and multiplier $1$. Since $z_c$ is not a cut point of $K(f_c)$, there is a unique attracting petal at $z_c$. Therefore, we have $$f_{c}^{\circ n}(z)=z+a_c(z-z_c)^2+O\left((z-z_c)^{3}\right)$$ for some $a_c \in \mathbb{C}^*$, as the Taylor series expansion of $f_{c}^{\circ n}$ near $z_c$. 

The situation for roots is a bit more complicated as they come in two different flavors. We say that a parabolic parameter $c$ in $\mathcal{M}_d$ is a \emph{primitive} root if $z_c$ is a cut point of $K(f_c)$ and there is a single attracting petal at $z_c$. A primitive root $c$ lies on the boundary of a unique hyperbolic component of $\mathcal{M}_d$. If the period of this unique hyperbolic component $H$ is $n$, then $f_c$ has an $n$-periodic parabolic cycle of multiplier $1$. As in the co-root case, the Taylor series of $f_{c}^{\circ n}$ near $z_c$ is given by $\left(z+a_c(z-z_c)^2+O\left((z-z_c)^{3}\right)\right)$ for some $a_c \in \mathbb{C}^*$. On the other hand, a parabolic parameter $c$ is called a \emph{satellite} root if $z_c$ is a cut point of $K(f_c)$ and there are at least two attracting petals at $z_c$. A satellite root $c$ is the (unique) common boundary point of two hyperbolic components $H$ and $H'$ of $\mathcal{M}_d$, one of which, say $H$, has $c$ as its root. Let the periods of the hyperbolic components $H$ and $H'$ be $n$ and $k$ respectively. Then the unique parabolic cycle of $f_c$ has period $k$ and multiplier a $q$-th root of unity, where $q=n/k\geq 2$. Moreover, the Taylor series expansion of $f_{c}^{\circ n}$ near $z_c$ is given by $\left(z+a_c(z-z_c)^{q+1}+O\left((z-z_c)^{q+2}\right)\right)$ for some $a_c \in \mathbb{C^*}$. In particular, there are $q$ attracting petals at the parabolic point $z_c$, and these petals are permuted transitively by $f_{c}^{\circ k}$. For proofs of these statements, see \cite[Lemma 17]{EMS}.

\section{Parabolic-like maps}\label{parlike}
The theory of parabolic-like maps extends the theory of polynomial-like maps to objects with a parabolic external class. For any polynomial map $P$ on the Riemann sphere $\widehat \C$, infinity is a superattracting fixed point, and the filled Julia set $K_P$ is the complement of the basin of attraction of infinity $\mathcal{A(\infty)}$, that is $K_P=\widehat \C \setminus \mathcal{A(\infty)}$. Thus, if $U$ is a suitable topological disk containing $K_P$ (for example, if the boundary of $U$ is a sufficiently large equipotential), then the preimage of $U$ is a topological disk $U'$ compactly contained in $U$, and $P_{|U'}:U'\rightarrow U$
 is a proper holomorphic map of degree $d=\mathrm{deg}(P)$. The triple $(P, U',U)$ is a (trivial) example of a polynomial-like map. Formally, a (degree $d$) polynomial-like map is a triple $(f, U',U)$ where $U'$ and $U$ are topological disks, $U'\subset \subset U$
 and $f: U'\rightarrow U$ is a (degree $d$) proper holomorphic map \cite{DH2}. The filled Julia set $K_f$ of a polynomial-like map is the set 
 of points which never leave $U'$ under iteration (for a polynomial $P$, this is just $K_P$). With any degree $d$ polynomial-like map, one can associate a degree $d$ 
 covering of the unit circle $h_f: \mathbb{S}^1 \rightarrow \mathbb{S}^1$ which encodes the dynamics of the polynomial-like map outside its filled Julia set. The map $h_f$ is called the \textit{external map} of the polynomial-like map $f$. The external map of a polynomial-like map is strictly expanding, with all periodic points repelling,
 and it is defined up to real-analytic diffeomorphisms of the circle. In this way a polynomial-like map can be considered as a union of two different dynamical systems:
 the filled Julia set $K_f$ and the external map $h_f$. By replacing the external map of a degree $d$ polynomial-like map with the map 
 $z \rightarrow z^d$ (which is an external map for a degree $d$ polynomial), Douady and Hubbard proved that every degree $d$ polynomial-like 
 map is hybrid equivalent to a polynomial of the same degree (where a hybrid equivalence is a quasiconformal conjugacy $\phi$
 with $\overline \partial \phi=0$ on $K_f$), and that this polynomial is unique if $K_f$ is connected.
 
A parabolic-like map is an object similar to a polynomial-like map, in the sense that it can be considered as the union of two different 
 dynamical systems: the filled Julia set and the external map \cite{Lu}. However, the external map of a parabolic-like map contains a parabolic fixed point, which complicates the setting considerably.
 \begin{figure}[hbt!]
\centering
\psfrag{A}{$\Omega$}
\psfrag{A'}{$\Delta'$}
\psfrag{gamma+}{$\gamma_+$}
\psfrag{gamma-}{$\gamma_-$}
\psfrag{AA}{$\Omega'$}
\psfrag{d}{$f:U' \stackrel{d:1}{\rightarrow}U$}
\psfrag{1}{$f:\Delta' \stackrel{1:1}{\rightarrow} \Delta$}
\psfrag{D}{$\Delta$}
\psfrag{E}{$z_0$}
\includegraphics[width= 10cm]{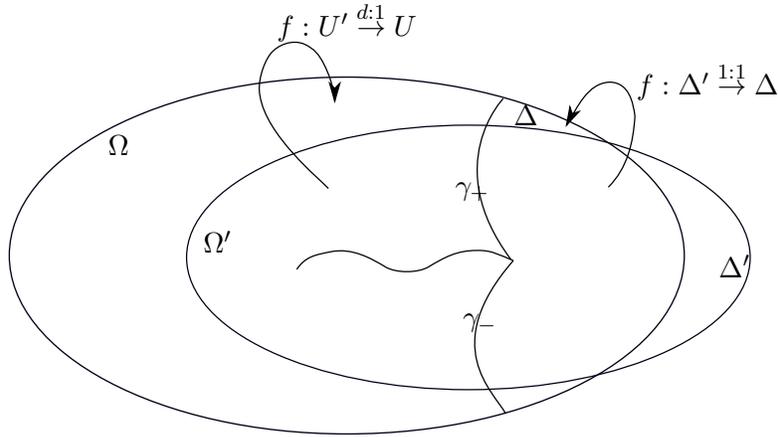}
\caption{\small For a parabolic-like map ($f,U',U,\gamma$), the arc $\gamma$ divides $U'$ and $U$
into $\Omega', \Delta'$
 and $\Omega, \Delta$ respectively. These sets are such that $\Omega'$ is compactly contained in $U$, $\Omega'\subset
\Omega$, $f:\Delta' \rightarrow \Delta$ is an isomorphism and $\Delta'$ contains at least one attracting fixed petal of the parabolic
fixed point.}
\label{AAA}
\end{figure}

\begin{definition}\label{definitionparlikemap} \textbf{(Parabolic-like maps)\,\,\,}
A \textit{parabolic-like map} of degree $d\geq2$ is a 4-tuple ($f,U',U,\gamma$) where 
\begin{itemize}
	\item $U'$ and $U$ are open subsets of $\C$, with $U',\,\, U$ and $U \cup U'$ isomorphic to a disc, and $U'$ not
contained in $U$,
	\item $f:U' \rightarrow U$ is a proper holomorphic map of degree $d\geq 2$ with a parabolic fixed point at $z=z_0$ of
 multiplier 1,
	\item $\gamma:[-1,1] \rightarrow \overline{U}$ is an arc with $\gamma(0)=z_0$, forward invariant under $f$, $C^1$
on $[-1,0]$ and on $[0,1]$, and such
that
$$f(\gamma(t))=\gamma(dt),\,\,\, \forall -\frac{1}{d} \leq t \leq \frac{1}{d},$$
$$\gamma([ \frac{1}{d}, 1)\cup (-1, -\frac{1}{d}]) \subseteq U \setminus U',\,\,\,\,\,\,\gamma(\pm 1) \in \partial U.$$
It resides in repelling petal(s) of $z_0$ and it divides $U'$ and $U$ into $\Omega', \Delta'$ and $\Omega, \Delta$
respectively, such that $\Omega' \subset \subset U$ 
(and $\Omega' \subset \Omega$), $f:\Delta' \rightarrow \Delta$ is an isomorphism (see Figure \ref{AAA}) and
$\Delta'$ contains at least one attracting fixed petal of $z_0$. We call the arc $\gamma$ a \textit{dividing arc}.

\end{itemize}
\end{definition}
The filled Julia set $K_f$ of a parabolic-like map $(f,U',U,\gamma)$ is the set of points which never leave $\Omega'\cup\{z_0\}$ under 
iteration. The model family in degree $2$ is given by the family of quadratic rational maps with a parabolic fixed point of multiplier $1$
at infinity (normalized by having critical points at $\pm1$), this is $Per_1(1)=\{[P_A] \,| \, P_A(z)=z+ 1/z+ A,\,\,A \in \C\}.$
The filled Julia set for $P_A$ with $A \neq 0$ is the complement of the parabolic basin of infinity $\mathcal{A}_A(\infty)$, so $K_A=\widehat \C \setminus \mathcal{A}_A(\infty)$ (for $A=0$ we need to make a choice, since both the left and right
half planes are parabolic basins. We set $K_{0}=\overline{\H}_l$). The map $h_2(z)= \frac{z^2+1/3}{z^2/3+1}$ is an external map for every $P_A,\,A\in \C$ \cite[Proposition 4.2]{Lu}. By replacing the external map of a degree $2$ parabolic-like map with the map $h_2$, the first author
proved \cite{Lu} that every degree $2$ parabolic-like map is hybrid equivalent to a member of the family $Per_1(1)$, and that this member is unique
if $K_f$ is connected. 
For more detailed studies on parabolic-like maps, consult \cite{Lu} for a dynamical description, \cite{Lu2} for a parameter space (of degree $2$ analytic families of parabolic-like maps) description, and \cite{Lu3} for an easy discussion on the results contained in the previous two articles. 

\section{Proof of the Theorem}\label{proof}

In the dynamical plane of the root of a satellite hyperbolic component of the multibrot set $\mathcal{M}_d$, a suitable iterate of the polynomial admits a degree $d$ parabolic-like restriction (see \cite[\S 3.1, Example 3]{Lu} for details of the construction in the case $d=2$, the case
$d>2$ being similar). The following lemma proves a rigidity statement for these parabolic-like maps.

\begin{lemma}[Rigidity of Parabolic-like Mappings]\label{rigidity_PL_maps}
Let $c_1$ and $c_2$ be the root points of two satellite hyperbolic components $H_1$ and $H_2$ (of period $n_1$ and $n_2$ respectively) of the Multibrot set $\mathcal{M}_d$.
If the parabolic-like mappings defined by the restrictions of $f_{c_1}^{\circ n_1}$ and $f_{c_2}^{\circ n_2}$ (around their critical value Fatou components) are conformally conjugate, then $c_1=c_2$ up to affine conjugacy.
\end{lemma}

\begin{proof}
Since $c_i$ is the root of a satellite component $H_i$ of period $n_i$, the polynomial $f_{c_i}$ has a unique parabolic cycle of period $k_i$ ($<n_i$) with multiplier a $q_i$-th root of unity, where $q_i=n_i/k_i$. The Taylor series expansion of $f_{c_i}^{\circ n_i}$ near $z_i$ is given by $\left(z+a_i(z-z_i)^{q_i+1}+O\left((z-z_i)^{q_i+2}\right)\right)$ for some $a_i\in\mathbb{C^*}$. In particular, there are $q_i$ attracting petals at the parabolic point $z_i$, and these petals are permuted transitively by $f_{c_i}^{\circ k_i}$. Let us start by noticing that if the parabolic-like mappings defined by the restrictions of $f_{c_1}^{\circ n_1}$ and $f_{c_2}^{\circ n_2}$ are conformally conjugate, then the parabolic germs of $f_{c_i}^{\circ n_i}$ at $z_i$ (for $i=1, 2$) are conformally conjugate. As the number of attracting petals of a parabolic germ is preserved by a topological conjugacy, it follows that $q_1=q_2=q$ (say).

We label the $q$ Fatou components of $f_{c_i}$ touching at the characteristic parabolic point $z_i$ counter-clockwise such that $U^1_i$ is the Fatou component of $f_{c_i}$ containing the critical value $c_i$. Since $c_i$ is the root of a satellite component with a $k_i$-periodic parabolic cycle, the polynomial $f_{c_i}^{\circ k_i}$ has a polynomial-like restriction ($h_i,V_i',V_i$) that is hybrid equivalent to some (degree $d$) $\frac{p_i}{q}$-rabbit (basilica if $q=2$) parameter on the boundary of the principal hyperbolic component of $\mathcal{M}_{d}$ (more precisely, $f_{c_i}^{\circ k_i}$ has a polynomial-like restriction $h_i$ that is hybrid equivalent to some polynomial $f_{c_i'}$ with a fixed point of multiplier $e^\frac{2\pi ip_i}{q}$). 

\begin{figure}[ht!]
\begin{center}
\includegraphics[scale=0.22]{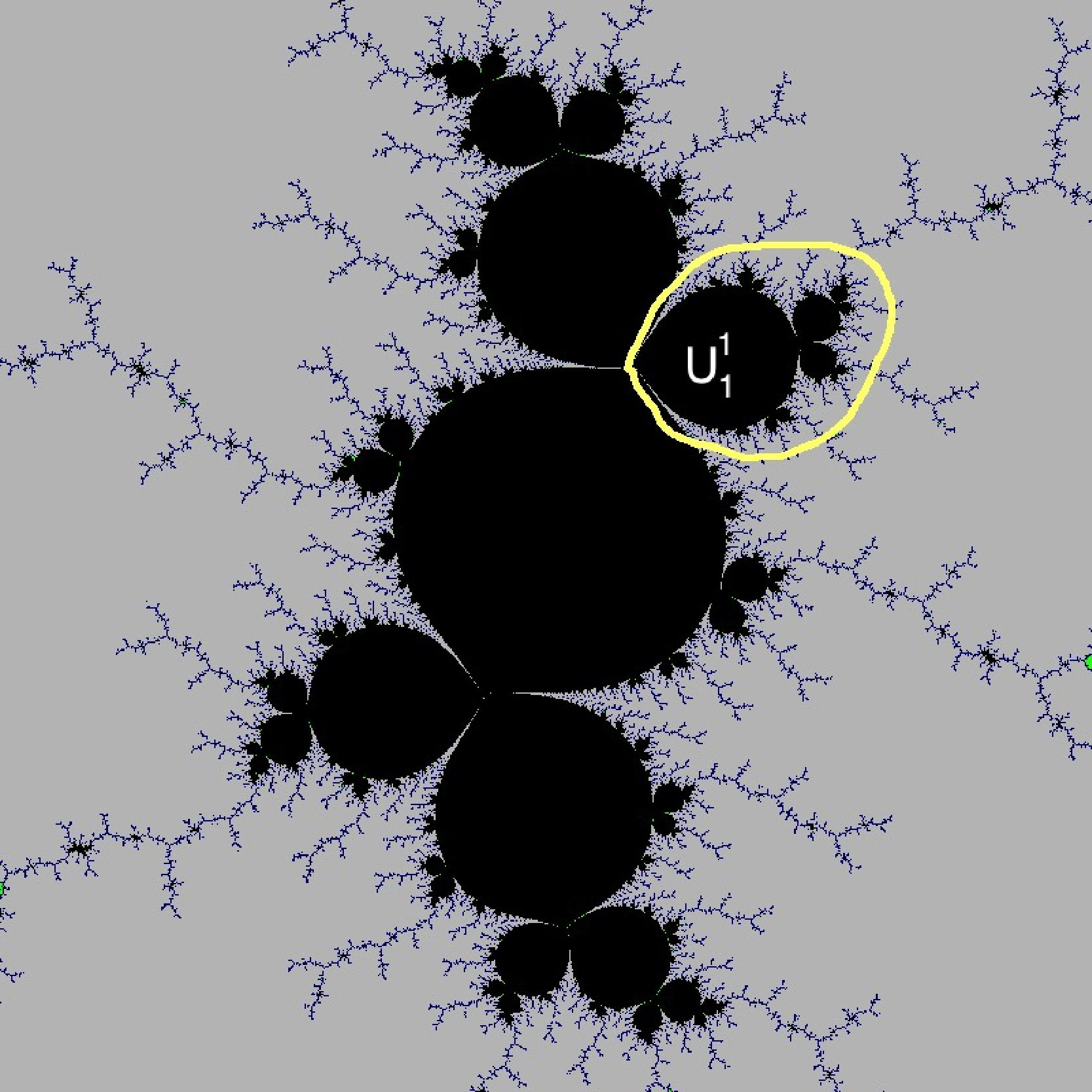}\ \includegraphics[scale=0.22]{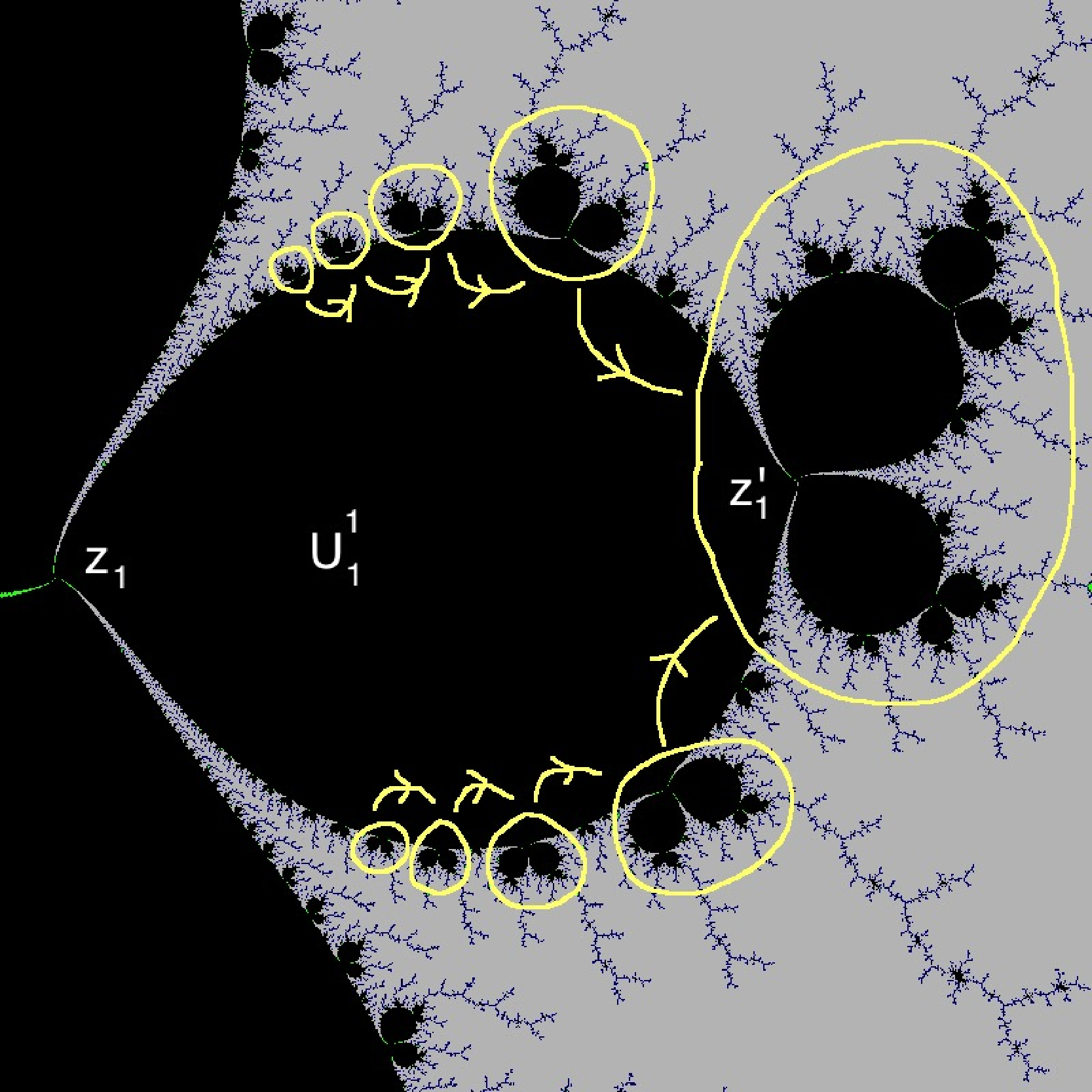}
\end{center}
\caption{Left: The characteristic ear $\mathcal{E}_1$ of $K(h_1)$ is enclosed by the yellow curve. Right: Under $f_{c_1}^{\circ n_1}$, each region enclosed by a yellow curve univalently maps to the `next' one. The initial conformal conjugacy between the two parabolic-like maps can be extended to a neighborhood of $\mathcal{E}_1$ using the univalent restrictions of $f_{c_1}^{\circ n_1}$ indicated in the figure.}
\label{germ_to_ear}
\end{figure}

Let $\eta$ be a conformal conjugacy between the parabolic-like restrictions of $f_{c_1}^{\circ n_1}$ and $f_{c_2}^{\circ n_2}$ in neighborhoods of $\overline{U_1^1}$ and $\overline{U_2^1}$ (respectively). A priori, $\eta$ is defined only in a neighborhood $W$ of $\overline{U_1^1}$. We can assume, possibly after shrinking $W$ (but ensuring that it still contains $\overline{U_1^1}$), that $f_{c_1}^{\circ n_1}$ has a unique critical point in $W$. We will now use the dynamics $f_{c_1}^{\circ n_1}$ to extend $\eta$ to a conformal conjugacy between $f_{c_1}^{\circ n_1}$ and $f_{c_2}^{\circ n_2}$ from a neighborhood of $K(h_1)$ to a neighborhood of $K(h_2)$.  

Let us define the ``characteristic ear'' $\mathcal{E}_1$ of $K(h_1)$ to be the closure of the connected component of $K(h_1)\setminus \{z_1\}$ containing the critical value $c_1$ (see Figure \ref{germ_to_ear} (Left)). Note that $f_{c_1}^{\circ n_1}(\mathcal{E}_1)=K(h_1)$. More precisely, there is a unique (strictly) pre-periodic point $z_1'$ on $\partial U_1^1$ such that $f_{c_1}^{\circ n_1}(z_1')=z_1$, and the closures $\mathcal{E}_2,\cdots,\mathcal{E}_q$ of the connected components of $\mathcal{E}_1\setminus\{z_1'\}$ \emph{not} containing $c_1$ are univalently mapped by $f_{c_1}^{\circ n_1}$ onto the closures of the connected components of $K(h_1)\setminus\{z_1\}$ \emph{not} containing $c_1$ (compare Figure \ref{ear}). Therefore, it suffices to first extend the conjugacy $\eta$ to a neighborhood of $\mathcal{E}_1$, and then use the above-mentioned univalent restrictions of $f_{c_1}^{\circ n_1}$ (on neighborhoods of $\mathcal{E}_2,\cdots,\mathcal{E}_q$) to spread it to a neighborhood of $K(h_1)$ (via the formula $f_{c_2}^{\circ n_2}\circ\eta\circ \left(f_{c_1}^{\circ n_1}\right)^{-1}$).

Since $K(h_1)$ is locally connected, the iterated pre-images of $\mathcal{E}_2\cup\cdots\cup\mathcal{E}_q$ under $f_{c_1}^{\circ n_1}$ (choosing suitable inverse branches of $f_{c_1}^{\circ n_1}$ such that the pre-images are attached to $\partial U_1^1$) shrink to the point $z_1$. Therefore, $W$ contains infinitely many such inverse images of $\mathcal{E}_2\cup\cdots\cup\mathcal{E}_q$. We can now iteratively extend $\eta$ to a neighborhood of $\mathcal{E}_1$ using the formula $f_{c_2}^{\circ n_2}\circ\eta\circ \left(f_{c_1}^{\circ n_1}\right)^{-1}$ (compare Figure \ref{germ_to_ear} (Right)).

\begin{figure}[ht!]
\begin{center}
\includegraphics[scale=0.22]{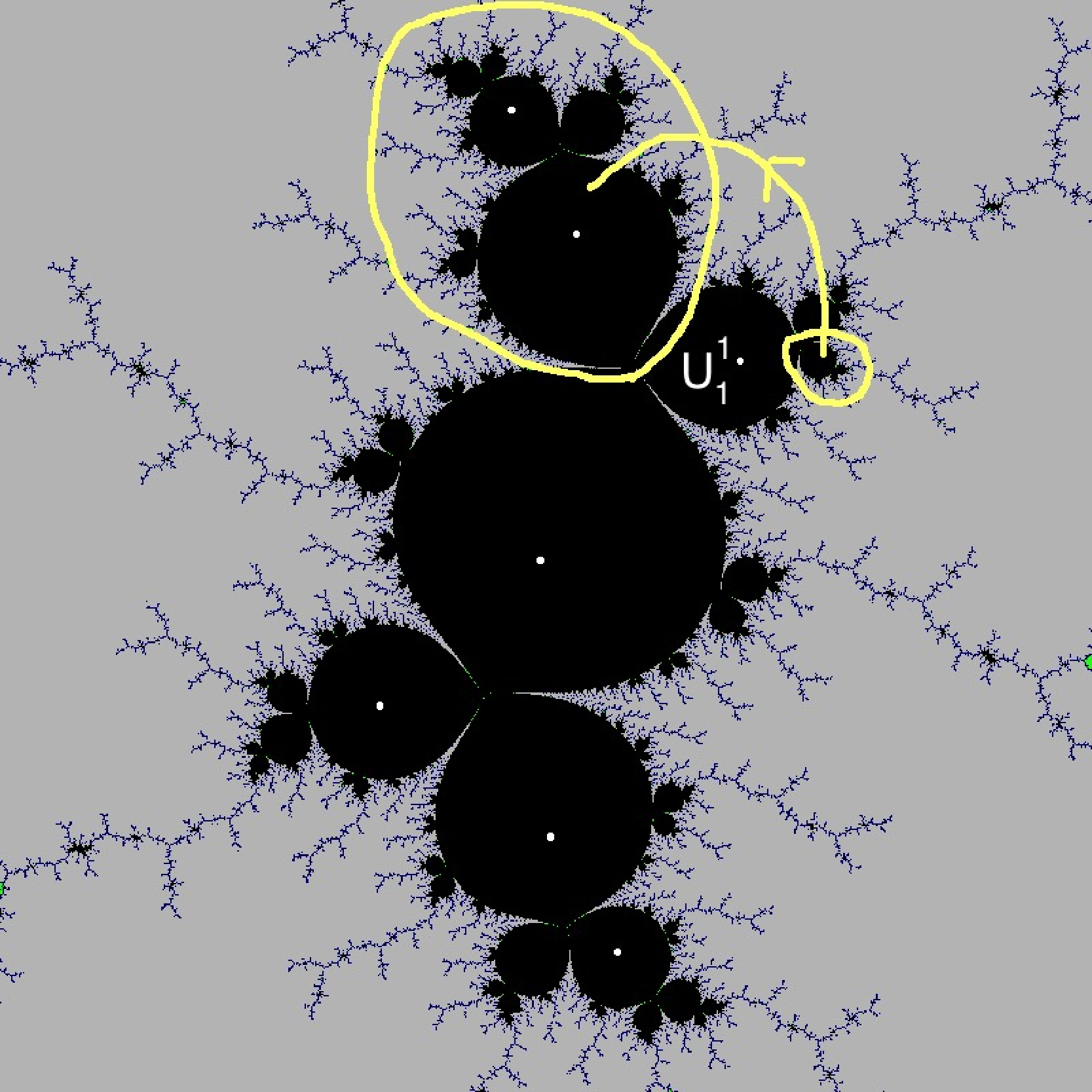}\ \includegraphics[scale=0.22]{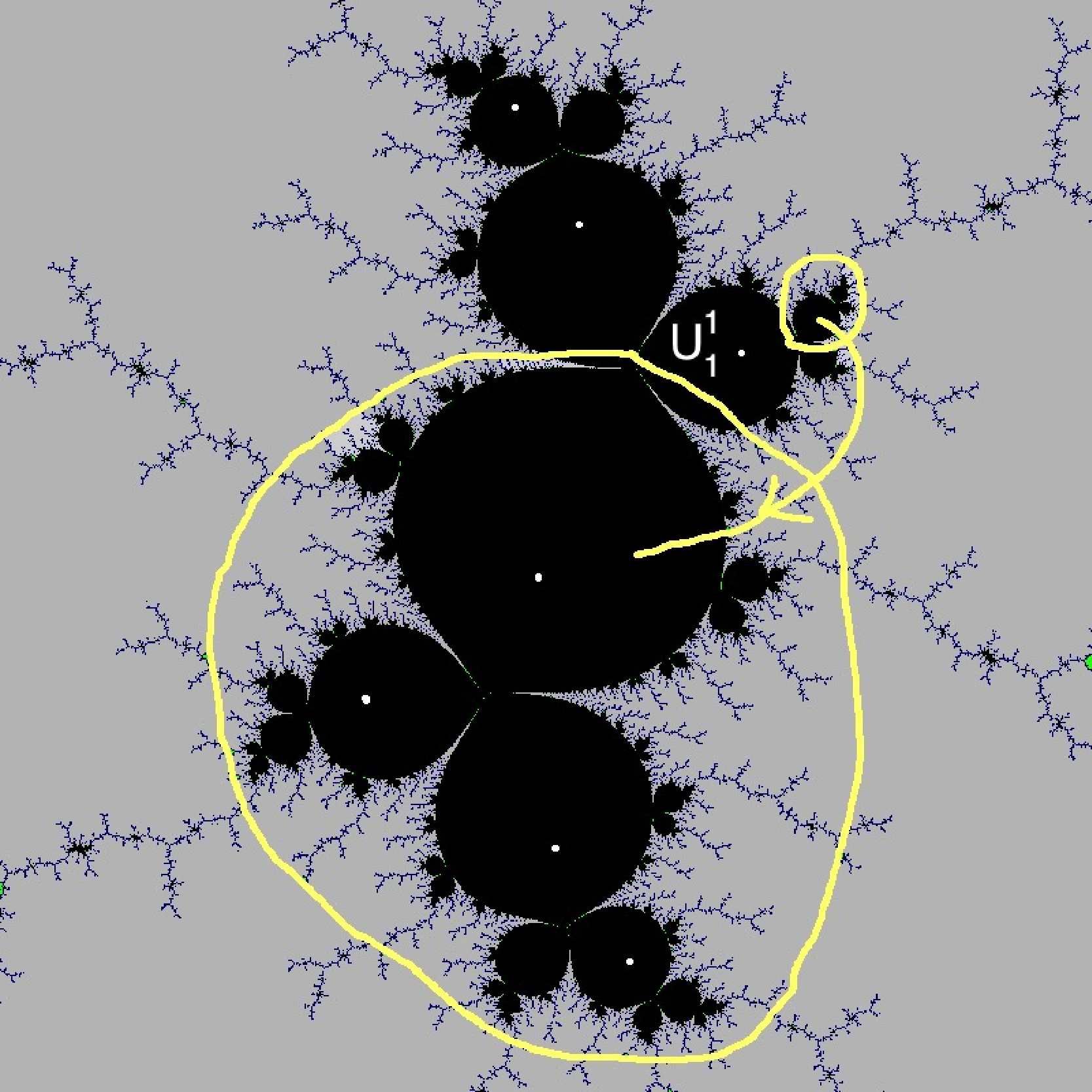}
\end{center}
\caption{Both figures show the small filled Julia set $K(h_1)$ in the dynamical plane of $f_{c_1}$, where $c_1$ is a satellite root of the Mandelbrot set with $q=3$. The critical points of $f_{c_1}^{\circ n_1}$ in $K(h_1)$ are marked. A small neighborhood of the characteristic ear $\mathcal{E}_1$ contains exactly one of these critical points. We obtain the required extension of $\eta$ by first extending it univalently to a neighborhood of the characteristic ear $\mathcal{E}_1$ and then using the univalent restrictions of $f_{c_1}^{\circ n_1}$ indicated in the figures to spread $\eta$ to a neighborhood of $K(h_1)$.}
\label{ear}
\end{figure}

To summarize, we have defined a (finite) sequence of maps $\{\eta_s\}_{s=0}^N$ such that $\eta_0$ is the original conjugacy $\eta$, and $\eta_s:= f_{c_2}^{\circ n_2}\circ\eta_{s-1}\circ \left(f_{c_1}^{\circ n_1}\right)^{-1}$ for $s=1, 2, \cdots, N$ (choosing suitable inverse branches). Moreover, $\mathrm{Dom}(\eta_s)\cap U_1^1\neq\emptyset$, for $s=1,2,\cdots,N$, and $\displaystyle \bigcup_{s=0}^{N} \mathrm{Dom}(\eta_s)\supset K(h_1)$. Since $f_{c_1}^{\circ n_1}$ fixes $U_1^1$ and $\eta_0$ is a conjugacy between $f_{c_1}^{\circ n_1}\vert_{U_1^1}$ and $f_{c_2}^{\circ n_2}\vert_{U_2^1}$, it follows from the construction that each $\eta_s$ extends the conformal map $\eta_0$ defined on $U_1^1$. Hence by uniqueness of analytic continuation, all these extensions match up to yield a conformal map $\eta$ defined on a neighborhood of $K(h_1)$ such that it conjugates $f_{c_1}^{\circ n_1}$ to $f_{c_2}^{\circ n_2}$.

Therefore, $\eta$ is a conformal conjugacy between the polynomial-like maps $h_1^{\circ q}$ ($=f_{c_1}^{\circ n_1}$) and $h_2^{\circ q}$ ($=f_{c_2}^{\circ n_2}$). By \cite[Corollary 10.2]{IM1}, we conclude that $f_{c_1}$ and $f_{c_2}$ are affinely conjugate. 
\end{proof}

We continue to work with parameters $c_1$ and $c_2$ that are root points of satellite hyperbolic components $H_1$ and $H_2$ (of period $n_1$ and $n_2$ respectively) of the Multibrot set $\mathcal{M}_d$. Our next lemma shows that a local conjugacy between the parabolic germs of $f_{c_1}^{\circ n_1}$ and $f_{c_2}^{\circ n_2}$ can be promoted to a conformal conjugacy between two suitable degree $d$ parabolic-like mappings. As in the proof of the previous lemma, we label the Fatou components of $f_{c_i}$ touching at the characteristic parabolic point $z_i$ counter-clockwise such that $U^1_i$ is the Fatou component of $f_{c_i}$ containing the critical value $c_i$. In order to investigate the consequences of a conformal conjugacy between two polynomial parabolic germs, we will need to use the concept of extended horn maps (see Section \ref{preliminars}). 

\begin{lemma}\label{equivalence}
Let $c_1$ and $c_2$ be the root points of two satellite hyperbolic components $H_1$ and $H_2$ (of period $n_1$ and $n_2$ respectively) of the Multibrot set $\mathcal{M}_d$, and $z_1$ and $z_2$ be the characteristic parabolic points of $f_{c_1}$ and $f_{c_2}$ (respectively). Then the following are equivalent.
\begin{itemize}
\item The degree $d$ parabolic-like mappings defined by the restrictions of $f_{c_1}^{\circ n_1}$ and $f_{c_2}^{\circ n_2}$ (with filled Julia set $\overline{U_1^1}$ and $ \overline{U_2^1}$ respectively) are conformally conjugate.

\item The (tangent-to-identity) parabolic germs given by the restrictions of $f_{c_1}^{\circ n_1}$ and $f_{c_2}^{\circ n_2}$ (around $z_1$ and $z_2$ respectively) are conformally conjugate.
\end{itemize}
\end{lemma}

\begin{proof}
Conformal conjugacy of the parabolic-like maps clearly implies conformal conjugacy of the corresponding germs. So we only need to show that when $g_1:=f_{c_1}^{\circ n_1}\vert_{N_{1}}$ and $g_2:=f_{c_2}^{\circ n_2}\vert_{N_{2}}$ are conformally conjugate by some local biholomorphism $\phi_1:N_{1}\to N_{2}$ (where $N_i$ is a small neighborhood of $z_i$), the degree $d$ parabolic-like maps $f_{c_1}^{\circ n_1}$ and $f_{c_2}^{\circ n_2}$ (with filled Julia set $\overline{U_1^1}$ and $\overline{U_2^1}$ respectively) are also conformally conjugate. We now proceed to prove this.

For $i=1,2$, let us suppose that the period of the characteristic parabolic point $z_i$ of $f_{c_i}$ be $k_i$. Then $f_{c_i}$ has $q_i=n_i/k_i$ attracting petals at $z_i$ which are permuted transitively by $f_{c_i}^{\circ k_i}$. Since two conformally conjugate germs have the same number of attracting petals, we have that $q_1=q_2=q$ (say). Note that $\phi_1$ must map an attracting petal $\mathcal{P}^{\textrm{att},1}_{c_1}\subset N_1\cap U_1^1$ to an attracting petal $\mathcal{P}^{\textrm{att},k}_{c_2}\subset N_2\cap U_2^k$ for some $k\in\{1,2,\cdots,q\}$. So, $\phi :=  f_{c_2}^{\circ k_2(1-k)}\circ\phi_1$ is a conformal conjugacy between $g_1$ and $g_2$ such that it maps $\mathcal{P}^{\textrm{att},1}_{c_1}$ to a petal  $\mathcal{P}^{\textrm{att}, 1}_{c_2}\subset N_2\cap U_2^1$.

For $k\in \mathbb{Z}/q\mathbb{Z}$, if $\psi^{\mathrm{att},k}_{c_2}$ is an extended attracting Fatou coordinate for $f_{c_2}^{\circ n_2}$ in $U_{2}^k$, then there exists an extended attracting Fatou coordinate $\psi^{\mathrm{att},k}_{c_1}$ for $f_{c_1}^{\circ n_1}$ in $U_{1}^k$ such that $\psi^{\mathrm{att},k}_{c_1}=\psi^{\mathrm{att},k}_{c_2}\circ \phi$ in their common domain of definitions. 

Similarly for $k\in \mathbb{Z}/q\mathbb{Z}$, if $\psi^{\mathrm{rep},k}_{c_2}$ is a repelling Fatou coordinate for $f_{c_2}^{\circ n_2}$ at $z_2$, then $\psi^{\mathrm{rep},k}_{c_1}:=\psi^{\mathrm{rep},k}_{c_2}\circ \phi$ is a repelling Fatou coordinate of $f_{c_1}^{\circ n_1}$ at $z_1$. The inverses of these repelling Fatou coordinates admit global extensions, and we denote them by $\Psi^{\mathrm{rep},k}_{c_i}$ (compare \cite[\S 2.5]{BE}).

Using the extended attracting and repelling Fatou coordinates described above, we can define extended horn maps $h^{+}_{c_i,k}:=\psi^{\mathrm{att},k}_{c_i}\circ\Psi^{\mathrm{rep},k}_{c_i}$ and $h^{-}_{c_i,k}:=\psi^{\mathrm{att},k+1}_{c_i}\circ\Psi^{\mathrm{rep},k}_{c_i}$ of $f_{c_i}^{n_i}$ at $z_i$. Since these extended attracting and repelling Fatou coordinates  are related by $\phi$ in neighborhoods of $z_i$, it follows that $h^\pm_{c_1,k}=h^\pm_{c_2,k}$ for $k=1, 2, \cdots, k$. By \cite[Proposition 4]{BE}, each $h^{+}_{c_i,1}$ is a ramified covering with a unique critical value $\Pi\left(\psi^{\mathrm{att},1}_{c_i}(c_i)\right)$, where $\Pi(Z)=e^{2\pi iZ}$. It follows that $\Pi\left(\psi^{\mathrm{att},1}_{c_1}(c_1)\right)=\Pi\left(\psi^{\mathrm{att},1}_{c_2}(c_2)\right)$. Therefore, $\psi^{\mathrm{att},1}_{c_1}(c_1)$$-\psi^{\mathrm{att},1}_{c_2}(c_2)$$=r\in \mathbb{Z}$. Thus, we can normalize the attracting Fatou coordinates such that $\psi^{\mathrm{att},1}_{c_1}(c_1)=0$ and $\psi^{\mathrm{att},1}_{c_2}(c_2)=-r$.

Let us define the map $\widetilde \psi^{\mathrm{att},1}_{c_2}:= \psi^{\mathrm{att},1}_{c_2} \circ g_2^{\circ r}$ on $\mathcal{P}^{\textrm{att},1}_{c_2}$. Clearly $\widetilde \psi^{\mathrm{att},1}_{c_2}$
is an attracting Fatou coordinate for $f_{c_2}^{\circ n_2}$ at $z_2$ such that for all $x \in \mathcal{P}^{\textrm{att},1}_{c_2}$, we have $\widetilde\psi^{\mathrm{att},1}_{c_2}(x)-\psi^{\mathrm{att},1}_{c_2}(x)=r$. By analytic continuation, we have an extended attracting Fatou coordinate $\widetilde \psi^{\mathrm{att},1}_{c_2}:U_{2}^1 \rightarrow \C$ for $f_{c_2}^{\circ n_2}$ such that $$\widetilde \psi^{\mathrm{att},1}_{c_2}(c_2)=\psi^{\mathrm{att},1}_{c_2}(c_2)+r=0=\psi^{\mathrm{att},1}_{c_1}(c_1).$$

Define $\eta:= (\widetilde \psi^{\mathrm{att},1}_{c_2})^{-1} \circ \psi^{\mathrm{att},1}_{c_1}: N_1\cap U_1^1 \rightarrow N_2 \cap U_2^1$, then $\eta$ is a conformal
conjugacy between $f_{c_1}^{\circ n_1}\vert_{N_1}$ and $f_{c_2}^{\circ n_2}\vert_{N_2}$ which extends by iterated lifting to a conformal conjugacy $\eta:U_{1}^1 \rightarrow U_{2}^1$ between $f_{c_1}^{\circ n_1}\vert_{U^1_1}$ and $f_{c_2}^{\circ n_2}\vert_{U^1_2}$.\footnote{Here is an alternative route to extend $\eta$ to the entire Fatou component. We can choose Riemann maps $\phi_{c_i}: U_{i}^1 \rightarrow \mathbb{D}$ with $\phi_{c_i}(c_i)=0$ such that $\phi_{c_i}$ conjugates $f_{c_i}^{\circ n_i}\vert_{U_i^1}$ to the Blaschke product $B(z)=\frac{3z^2+1}{3+z^2}$. An easy computation in Fatou coordinates now shows that $\phi_{c_2}^{-1}\circ \phi_{c_1}$ extends the local conjugacy $\eta$ to the entire immediate basin $U_1^1$ such that it conjugates $f_{c_1}^{\circ n_1}$ on $U_1^1$ to $f_{c_2}^{\circ n_2}$ on $U_2^1$.} Abusing notation, we will denote this extended conjugacy by $\eta$. Since the basin boundaries are locally connected, by Caratheodory's theorem the conformal conjugacy $\eta$ extends as a homeomorphism from $\partial U_{1}^1$ onto $\partial U_{2}^1$. Note also that by definition, $\eta=g_2^{\circ (-r)}\circ\phi$ in their common domain of definition. Therefore, $\eta$ admits an analytic continuation to a neighborhood $V$ of the point $z_1$, and continues to be a conjugacy between the germs $g_1$ and $g_2$. 

We will now extend this conformal conjugacy to a conformal conjugacy $\eta$ between a neighborhood of $\overline{U_1^1}$ and a neighborhood of $\overline{U_2^1}$. By Montel's theorem, we have $\displaystyle \bigcup_{s\in \mathbb{N}} f_{c_1}^{\circ sn_1}\left( V \cap \partial U_{1}^1 \right) = \partial U_{1}^1$. Since none of the $f_{c_1}^{\circ sn_1}$ has a critical point on $\partial U_{1}^1$, we can extend $\eta$ to a neighborhood of each point of $\partial U_{1}^1$ by using the functional equation $\eta \circ f_{c_1}^{\circ sn_1} = f_{c_2}^{\circ sn_2} \circ \eta$. Since all of these extensions at various points of $\partial U_{1}^1$ extend the already defined (and conformal) common map $\eta$, uniqueness of analytic continuation yields
an analytic extension of $\eta$ to a neighborhood of $\overline{U_{1}^1}$. By construction, this extension is clearly a proper holomorphic map, and assumes every point in $U_2^1$ precisely once. Therefore, the extended $\eta$ from a neighborhood of $\overline{U_{1}^1}$ onto a neighborhood of $\overline{U_{2}^1}$ has degree one. So $\eta$ is a conformal conjugacy between $f_{c_1}^{\circ n_1}$ and $f_{c_2}^{\circ n_2}$. This shows that the parabolic-like mappings defined by $f_{c_1}^{\circ n_1}$ and $f_{c_2}^{\circ n_2}$ in neighborhoods of $\overline{U_1^1}$ and $\overline{U_2^1}$ (respectively) are conformally conjugate. 
\end{proof}

We are now ready to prove the main result of this paper.

\begin{proof}[Proof of Theorem \ref{Rigidity_Satellite}]
By hypothesis, $c_1$ (respectively $c_2$) is the root or a co-root point of a hyperbolic component $H_1$ (respectively $H_2$) of period $n_1$ (respectively $n_2$) of $\mathcal{M}_d$. Also, $z_1$ (respectively $z_2$) is the characteristic parabolic point of $f_{c_1}$ (respectively $f_{c_2}$). For $i=1,2$, the Taylor series expansion of $f_{c_i}^{\circ n_i}$ around $z_i$ is given by $\left(z+a_i(z-z_i)^{q_i+1}+O\left((z-z_i)^{q_i+2}\right)\right)$ for some $q_i\geq 1$ and $a_i \in \mathbb{C^*}$. Note that $q_i$ is the number of attracting petals of $f_{c_i}$ at $z_i$. 

We assume that there exist neighborhoods $N_1$ and $N_2$ of $z_1$ and $z_2$ such that the parabolic germs $f_{c_1}^{\circ n_1}\vert_{N_1}$ and $f_{c_2}^{\circ n_2}\vert_{N_2}$ are conformally conjugate. Evidently, such a conjugacy implies that $q_1=q_2$. We will now consider two cases.\\

\noindent\textbf{Case 1: ($q_1=q_2=1$).}
In this case, each $c_i$ is a co-root or a primitive root of $H_i$ (compare Section \ref{preliminars}). Hence the unique parabolic cycle of $f_{c_i}$ has period $n_i$; i.e. the period of $z_i$ is $n_i$. Since the germs $f_{c_1}^{\circ n_1}\vert_{N_1}$ and $f_{c_2}^{\circ n_2}\vert_{N_2}$ are conformally conjugate, it follows by \cite[Theorem 1.4]{IM1} that $f_{c_1}$ and $f_{c_2}$ are affinely conjugate.\\

\noindent\textbf{Case 2: ($q_1=q_2>1$).}
In this case, each $c_i$ is a satellite root of $H_i$. Since the germs $f_{c_1}^{\circ n_1}\vert_{N_1}$ and $f_{c_2}^{\circ n_2}\vert_{N_2}$ are conformally conjugate, it follows by Lemma \ref{equivalence} that the degree $d$ parabolic-like mappings defined by the restrictions of $f_{c_1}^{\circ n_1}$ and $f_{c_2}^{\circ n_2}$ (around the critical value Fatou components $U_1^1$ and $U_2^1$ respectively) are conformally conjugate. By Lemma \ref{rigidity_PL_maps}, we can now conclude that $f_{c_1}$ and $f_{c_2}$ are affinely conjugate. 
\end{proof}
\bibliographystyle{alpha}
\bibliography{germs}
 
\end{document}